 \documentclass[letterpaper,11pt]{article}

\usepackage{graphicx}
\usepackage{etoolbox}
\newtoggle{anonymous}
\togglefalse{anonymous}

\usepackage[usenames,dvipsnames,svgnames,table]{xcolor}
\usepackage[colorlinks=true,linkcolor=blue,citecolor=ForestGreen]{hyperref}

\usepackage{amsmath,amsthm,amssymb,enumerate}
\usepackage{scrextend}
\usepackage[utf8]{inputenc}
\usepackage[T1]{fontenc}
\usepackage{orcidlink}
\usepackage{paralist}

\usepackage{amsmath}
\usepackage{amssymb}
\usepackage{bm}
\usepackage{thmtools}
\usepackage{hyperref}
\usepackage[capitalize]{cleveref}
\usepackage{float}
\usepackage[color=green!40,textsize=small]{todonotes}

\usepackage{scalefnt}
\usepackage{pifont}
\usepackage{xspace}

\usepackage{float}
\newcommand{\xmark}{\ding{55}}
\newcommand{\ques}{{\fontfamily{cyklop}\selectfont \emph{?}}}

\newcommand{\dec}{decent\xspace}

\newcommand{\Slin}{moderate-growth\xspace}
\newcommand{\Slinn}{moderately-growing\xspace} 

\newcommand{\Smull}{sub-multiplicative\xspace}
\newcommand{\Smul}{sub-multiplicativity\xspace}

 \usepackage[margin=1in]{geometry}

\usepackage{setspace}
\allowdisplaybreaks

\Crefname{figure}{Figure}{Figures}
\Crefname{claim}{Claim}{Claims}

\crefformat{equation}{\textup{#2(#1)#3}}
\crefrangeformat{equation}{\textup{#3(#1)#4--#5(#2)#6}}
\crefmultiformat{equation}{\textup{#2(#1)#3}}{ and \textup{#2(#1)#3}}
{, \textup{#2(#1)#3}}{, and \textup{#2(#1)#3}}
\crefrangemultiformat{equation}{\textup{#3(#1)#4--#5(#2)#6}}%
{ and \textup{#3(#1)#4--#5(#2)#6}}{, \textup{#3(#1)#4--#5(#2)#6}}{, and \textup{#3(#1)#4--#5(#2)#6}}

\Crefformat{equation}{Equation~\textup{#2(#1)#3}}
\Crefrangeformat{equation}{Equations~\textup{#3(#1)#4--#5(#2)#6}}
\Crefmultiformat{equation}{Equations~\textup{#2(#1)#3}}{ and \textup{#2(#1)#3}}
{, \textup{#2(#1)#3}}{, and \textup{#2(#1)#3}}
\Crefrangemultiformat{equation}{Equations~\textup{#3(#1)#4--#5(#2)#6}}%
{ and \textup{#3(#1)#4--#5(#2)#6}}{, \textup{#3(#1)#4--#5(#2)#6}}{, and \textup{#3(#1)#4--#5(#2)#6}}

\usepackage{tikz}
\tikzset{normalnode/.style={circle, draw, fill=black, inner sep=0, minimum width=1.5mm}}

\usetikzlibrary{fit}
\usetikzlibrary{shapes,arrows,shadows,positioning}
\usetikzlibrary{backgrounds}

\newcommand{\Nn}{\mathbb{N}}

\newcommand{\Rrr}{\mathbb{R}_{\geq 0}}

\newcommand{\Ac}{\mathcal{A}}
\newcommand{\Fc}{\mathcal{F}}
\newcommand{\Gc}{\mathcal{G}}
\newcommand{\Mc}{\mathcal{M}}

\newcommand{\Xc}{\mathcal{X}}
\newcommand{\Yc}{\mathcal{Y}}

\newtheorem{theorem}{Theorem}[section]
\newtheorem{question}{Question}
\newtheorem{proposition}[theorem]{Proposition}

\newtheorem{corollary}[theorem]{Corollary}
\newtheorem{conjecture}[theorem]{Conjecture}

\newtheorem{lemma}[theorem]{Lemma}

\theoremstyle{definition}
\newtheorem{definition}[theorem]{Definition}

\theoremstyle{remark}

\newcommand{\her}{\mathrm{Her}}
\newcommand{\mon}{\mathrm{Mon}}

\newcommand{\aut}{\mathrm{aut}}

\newcommand{\sub}{\#\mathrm{Sub}}
\newcommand{\emb}{\#\mathrm{Emb}}

\renewcommand{\leq}{\leqslant}
\renewcommand{\geq}{\geqslant}

\renewcommand{\Pr}[1]{\mathbb{P}\left[\,#1\,\right]}

\newcommand{\comment}[1]{}

\title{Tight Bounds on Adjacency Labels for Monotone Graph Classes}

\date{}

 \iftoggle{anonymous}{
 	\author{ }
 }{%
	\author{{\'E}douard Bonnet\thanks{Univ. Lyon, ENS de Lyon, UCBL, CNRS, LIP, France,  \texttt{edouard.bonnet@ens-lyon.fr}, \orcidlink{0000-0002-1653-5822}}    \and Julien Duron\thanks{Univ. Lyon, ENS de Lyon, UCBL, CNRS, LIP, France, \texttt{julien.duron@ens-lyon.fr}, \orcidlink{0009-0004-0925-9438}} \and    John Sylvester\thanks{Department of Computer Science, University of Liverpool, UK, \texttt{john.sylvester@liverpool.ac.uk}, \orcidlink{0000-0002-6543-2934}}
	\and
	Viktor Zamaraev\thanks{Department of Computer Science, University of Liverpool, UK, \texttt{viktor.zamaraev@liverpool.ac.uk}, \orcidlink{0000-0001-5755-4141}}
	\and Maksim Zhukovskii\thanks{Department of Computer Science, University of Sheffield, UK, \texttt{m.zhukovskii@sheffield.ac.uk}, \orcidlink{0000-0001-8763-9533} }}
}

\begin{document}


\maketitle

\begin{abstract}
  A class of graphs admits an adjacency labeling scheme of size $b(n)$, if the vertices in each of its $n$-vertex graphs can be assigned binary strings (called labels) of length $b(n)$ so that the adjacency of two vertices can be determined solely from their labels.
  
  We give tight bounds on the size of adjacency labels for every family of monotone (i.e., subgraph-closed) classes with a well-behaved growth function between $2^{O(n \log n)}$ and $2^{O(n^{2-\delta})}$ for any $\delta > 0$.
  Specifically, we show that for any function $f: \mathbb N \to \mathbb R$ satisfying $\log n \leqslant f(n) \leqslant n^{1-\delta}$ for any fixed $\delta > 0$, and some~sub-multiplicativity condition, there are monotone graph classes with growth $2^{O(nf(n))}$ that do not admit adjacency labels of size at most $f(n) \log n$.
  On the other hand, any such class does admit adjacency labels of size $O(f(n)\log n)$.
  Surprisingly this tight bound is a $\Theta(\log n)$ factor away from the information-theoretic bound of~$\Omega(f(n))$.
  The special case when $f = \log$ implies that the recently-refuted Implicit Graph Conjecture [Hatami and Hatami, FOCS 2022] also fails within monotone classes. 
  
We further show that the Implicit Graph Conjecture holds for all monotone \emph{small} classes. In other words, any monotone class with growth rate at most $n!\,c^n$ for some constant $c>0$, admits adjacency labels of information-theoretic order optimal size. In fact, we show a more general result that is of independent interest: any monotone small class of graphs has bounded degeneracy.
  We conjecture that the Implicit Graph Conjecture holds for all hereditary small classes.
\end{abstract}

	\section{Introduction}
	\label{sec:new-intro}
	A \emph{class} of graphs is a~set of graphs which is closed under isomorphism. 
	For a class of graphs $\Xc$ we denote by $\Xc_n$ the set of graphs in $\Xc$ with vertex set $[n]$.
	The function $n \mapsto |\Xc_n|$ is called the \emph{speed} of $\Xc$. 
	A \emph{coding} of graphs is a representation of graphs by words in the binary alphabet $\{ 0, 1 \}$.  One of the main considerations with graph representations is their succinctness; clearly, any representation of $n$-vertex graphs in a class $\Xc$ would require at least $\lceil \log |\Xc_n| \rceil$ bits for some graphs in $\Xc_n$.	
	
	Another consideration is whether the representation is global or local. Standard graph representations, such as adjacency matrix or adjacency lists, are examples of \emph{global} representations, where a graph is stored in a single data structure that needs to be accessed in order to query some information about the graph, e.g., adjacency between a pair of vertices. By contrast, in \emph{local} graph representations, the  encoding of a graph is distributed over its vertices in such a way that the queries can be answered by looking only into the local information associated with the vertices involved in the query. In this work we are concerned with local graph representations for adjacency queries, i.e., queries that given two vertices answer whether they are adjacent or not.
	
	Let $\Xc$ be a~class of graphs and $b : \Nn \rightarrow \Nn$ be a~function. A \emph{$b(n)$-bit adjacency labeling scheme} (or simply \emph{$b(n)$-bit labeling scheme})  for $\Xc$ is a~pair (encoder, decoder) of algorithms where for any $n$-vertex graph $G\in \Xc_n$ the encoder assigns binary strings, called \emph{labels}, of length $b(n)$ to the vertices of $G$ such that the adjacency between any pair of vertices can be inferred by the decoder only from their labels.
	We note that the decoder depends on the class $\Xc$, but not on the graph~$G$.
	The function~$b(\cdot)$ is the \emph{size} of the labeling scheme.
	Adjacency labeling schemes were introduced by Kannan, Naor, and Rudich \cite{KNR88,KNR92},
	and independently by Muller \cite{Muller88} in the late 1980s and have been actively studied since then.
	%
	Adjacency labeling schemes are closely related to induced universal graphs, which we will refer to simply as universal graphs.
	For a~function $u : \Nn \rightarrow \Nn$, a~\emph{universal graph sequence} or simply \emph{universal graph} of size $u(n)$ is a~sequence
	of graphs $(U_n)_{n\in \Nn}$ such that for every $n \in \Nn$ the graph $U_n$ has at most $u(n)$ vertices and every $n$-vertex graph in $\Xc$ is an induced subgraph of~$U_n$.
	It was observed in \cite{KNR92} that for a~class of graphs the existence of a $b(n)$-bit labeling scheme is equivalent to the existence of a~universal graph of size $2^{b(n)}$.
	
	The binary word, obtained by concatenating labels of the vertices of a graph $G \in \Xc_n$ assigned by an adjacency labeling scheme, uniquely determines graph $G$.
	Thus, a $b(n)$-bit labeling scheme cannot represent more than $2^{n b(n)}$ graphs on $n$ vertices, and therefore, if $\Xc$ admits a $b(n)$-bit labeling scheme, then $|\Xc_n| \leq 2^{n b(n)}$.
	This implies a lower bound of $\frac{\log |\Xc_n|}{n}$ on the size $b(n)$ of any adjacency labeling scheme for $\Xc$.
	A natural and important question is: which classes admit an adjacency labeling scheme of a size that matches this information-theoretic lower bound?
	
	We say that a graph class $\Xc$ admits an \emph{implicit representation}, if it admits an information-theoretic \emph{order optimal} adjacency labeling scheme, i.e.,\ if $\Xc$ has
	a $b(n)$-bit labeling scheme, where $b(n) = O(\log |\Xc_n|/n)$. 
	Equivalently, $\Xc$ admits an implicit representation if $\Xc$ has a universal graph of size $\exp(O(\log |\Xc_n|/n))$.
	For example, the class $\Ac$ of all graphs admits an implicit representation, because
        \[|\Ac_n| = 2^{\binom{n}{2}} = 2^{\Theta(n^2)}~\text{and}~b(n) = O\left(\frac{\log |\Ac_n|}{n}\right) = O(n),\]
        and one can easily design an $O(n)$-bit labeling scheme for $\Ac$, e.g., by assigning to each vertex of a graph an $(n + \lceil \log n \rceil)$-bit label consisting of the row in an adjacency matrix of the graph corresponding to the vertex and the index of that row; in fact, as we discuss below, the class of all graphs admits an asymptotically optimal $(1+o(1))n/2$-bit labeling scheme \cite{A17}.
	
	However, not every class admits an implicit representation. The following example is due to Muller \cite{Muller88} (see also \cite{Spinrad03}). Let $\Yc$ be the class of graphs in which the number of edges
	does not exceed the number of vertices. It is easy to estimate that $|\Yc_n| = 2^{O(n \log n)}$. To show that this class
	does not admit an implicit representation, consider an arbitrary $n$-vertex graph $G$. Obviously, $G$ does not necessarily belong to $\Yc$, but after adding $n^2-n$ isolated vertices to $G$, we obtain a graph $H$ on $N=n^2$ vertices that belongs to $\Yc$.
	Now, if an $O(\log n)$-bit labeling scheme for $\Yc$ existed, then the $O(\log N)$-bit adjacency labels for $H$ could be used
	as $O(\log n)$-bit adjacency labels for $G$. Since, $G$ was chosen arbitrarily, this is in contradiction with the lower bound of 
	$\frac{\log |\Ac_n|}{n} = \Omega(n)$ on the size of any labeling scheme for the class $\Ac$ of all graphs.
	
	The crucial property used in the above example is that by adding isolated vertices to a graph not in $\Yc$, one can obtain a graph in $\Yc$.
	Using more familiar terminology, one would say that class $\Yc$ is not \emph{hereditary}, i.e., it is not closed under vertex removal or, equivalently, under taking induced subgraphs. Many natural graph classes (e.g.,\ forests, planar graphs, bipartite graphs, geometric intersection graphs) are hereditary. 
	It turns out that finding a hereditary graph class that does not admit an implicit representation is a non-trivial question. The first instance of this question was asked by 	Kannan, Naor, and Rudich \cite{KNR88} for \textit{factorial classes} (i.e.,~graph classes $\mathcal{X}$ with the speed $|\mathcal{X}_n|=2^{O(n\log n )}$), which was later stated by Spinrad \cite{Spinrad03} in the form of a~conjecture, that became known as the \emph{Implicit Graph Conjecture}.
	
	\begin{labeling}{(\textit{IGC}):}
		\item[(\textit{IGC}):] Any hereditary graph class of at most factorial speed admits an $O(\log n)$-bit labeling scheme.
	\end{labeling}

	This question remained open for over 30 years until
	a recent breakthrough by Hatami and Hatami \cite{HH22}.
	They showed that, for any $\delta > 0$, there exists a~hereditary factorial class that does not admit a labeling scheme of size $n^{1/2-\delta}$, which is very far from the information-theoretic lower bound of $\Omega(\log n)$.
	This result leaves wide open the question of characterizing factorial hereditary graph classes that admit an implicit representation 
	(see e.g. \cite{HWZ22} for more discussion).

	Factorial hereditary classes form an important family, as many classes of theoretical or practical interest are factorial (e.g.,~forests, planar graphs, disk graphs, graphs of bounded twin-width). However, as was noted by Spinrad \cite{Spinrad03}, there is nothing that prevents one from considering implicit representability of other hereditary graph classes. Spinrad \cite{Spinrad03} raised this as the \textit{Generalized Implicit Graph Question}, which we restate using the terminology of our paper as follows.
	
	\begin{question}[\cite{Spinrad03}]\label{q:gen-igq}
		Which hereditary graph classes admit implicit representations?
	\end{question}

	The answer to this question is known for classes with $|\Xc_n| = 2^{\Omega(n^2)}$, and for \emph{subfactorial} graph classes, i.e., classes $\Xc$ with $|\Xc_n| = 2^{o(n\log n)}$. Indeed, for the latter classes, it is known that they have at most exponential speed, i.e., $|\Xc_n| = 2^{O(n)}$ \cite{Alekseev97,SZ94}, and also admit $O(1)$-bit labeling schemes \cite{Sch99}.
	For the former classes, the $O(n)$-bit labeling scheme mentioned above for the
	class $\Ac$ of all graphs is an order optimal labeling scheme.
	In fact, in this regime, \emph{information-theoretic asymptotically optimal} (up to the second-order term) labeling schemes are available. For the class of all graphs, such results (in the language of universal graphs) were available since 1965 \cite{moon65,AKTZ15,A17}. 
	For proper hereditary graph classes $\Xc$ with the speed $2^{\Omega(n^2)}$,
	by the Alekseev–Bollobás–Thomason theorem \cite{Ale92,BT95}, their speed is $|\Xc_n| = 2^{(1-1/k(\Xc))n^2/2 + o(n^2)}$, where $k(\Xc)$ is an integer greater than 1. Recently, Bonamy, Esperet, Groenland, and Scott showed \cite{BEGS21} that all such classes have asymptotically optimal adjacency labeling schemes of size $(1-1/k(\Xc))n/2 + o(n)$.
	
	For the classes in the intermediate range, i.e.,~the classes with the speed between $2^{\Omega(n\log n)}$ and $2^{o(n^2)}$ the picture is much less understood (see \cref{fig:WorldMap}).	
	Most known information is concentrated on the lower extreme of the range, i.e.,~around factorial speed, which was promoted by the Implicit Graph Conjecture.
	Factorial graph classes from certain families are known to admit implicit representations:
 	proper minor-closed graph classes \cite{GL07}, graph classes of bounded degeneracy (equivalently, of bounded arboricity) \cite{KNR88}, clique-width \cite{CV03,Spinrad03} (see also \cite{ban22}), and twin-width \cite{BGK22} all admit implicit representations.
	The only lower bound witnessing (non-constructively) factorial classes\footnote{This lower bound is sufficiently large to rule out the existence of implicit representations even for hereditary classes of size $2^{\Theta(n^{3/2-\delta})}$, for any fixed $0<\delta<1/2$.} that do not admit an implicit representation is the above-mentioned result by Hatami and Hatami \cite{HH22}. A notable family of 
	hereditary graph classes where \cref{q:gen-igq} remains open is the \emph{small} graph classes, i.e., classes $\Xc$ with $|\Xc_n| \leq n!\,c^n$ for some constant $c$. These classes encompass only the bottom part of the factorial layer
	and include proper minor-closed classes \cite{BlankenshipThesis,NSTW06}, and more generally, classes of bounded twin-width \cite{BGK22}. However, it is still  unknown if all such classes admit an implicit representation (see \cite{BDSZZ23} for more details on implicit representation of small classes).
	Alon showed \cite{Alon23} that every hereditary graph class $\Xc$ with $|\Xc_n| = 2^{o(n^2)}$ admits an $n^{1-\delta}$-bit labeling scheme for some $\delta > 0$.

\subsection{Our contribution} In this paper, we study \cref{q:gen-igq} for \emph{monotone} graph classes, i.e., graph classes that are closed under taking subgraphs. Monotone graph classes form a subfamily of hereditary graph classes. Together with some previous results mentioned in the introduction, the results of this paper give a~near complete resolution of \cref{q:gen-igq} for monotone classes. We state our results below.

	The \emph{degeneracy} of a graph $G$ is the minimum $d$ such that every subgraph of $G$
        has a vertex of degree at~most~$d$.
        We say that a class of graphs $\Xc$ has bounded 
        degeneracy, if there exists a~constant~$d$ such that the degeneracy of every graph $G \in \Xc$ is at~most~$d$; otherwise, we say that $\Xc$ has unbounded degeneracy.
        Our first main result shows that degeneracy is bounded for monotone small classes.

\begin{figure}[!t]
	\begin{center}
		
		\includegraphics[width=0.83\textwidth]{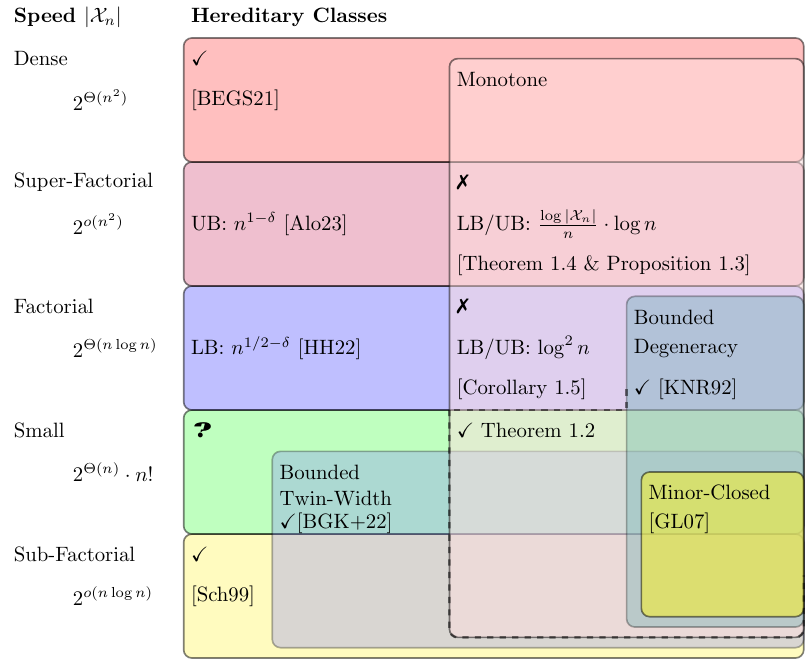}
	\end{center}
	\caption{A $\checkmark$ indicates that all classes of the given type have an implicit representation, a \xmark\xspace shows that they do not, and a $\ques$ signals that the question is open. A $\checkmark$ is inherited by every sub-region, a \xmark\xspace is inherited to the left of the marked region, and a $\ques$ only holds in that region. The upper and lower bounds (UB and LB respectively) are stated  up to constants which may depend on the class.
        The dashed extension of the \emph{bounded degeneracy} region illustrates its containment of~monotone small classes (\cref{cor:monSmall-ir}).}\label{fig:WorldMap}
\end{figure}

 \def\smallDgn{
	Let $\Xc$ be a monotone small class. Then, $\Xc$ has bounded degeneracy.}
\begin{theorem}\label{th:smallDegenerate}
	\smallDgn
\end{theorem}

\cref{th:smallDegenerate} has wider reaching implications than just labeling schemes, and is of independent interest. 
In the context of \cref{q:gen-igq}, we obtain the following result from \cref{th:smallDegenerate} and a~classical labeling scheme for classes of bounded degeneracy \cite{KNR88} (see also \cref{lem:degenlabel}).
 
\def\smallcor{ Any monotone small class admits an implicit representation. }

\begin{theorem}\label{cor:monSmall-ir}
	\smallcor	
\end{theorem}

\noindent
This answers Question 1 from \cite{BDSZZ23} for monotone graph classes.

We now turn to monotone classes that are not small.
Our next result shows that any monotone class with non-decreasing speed admits a labeling scheme of size at most $O(\log n)$ away from the information-theoretic lower bound.

\def\monofactorial{		Let $f : \Rrr \rightarrow \Rrr$ be a non-decreasing function. Then, any monotone class of graphs $\mathcal{X}$ with speed $|\Xc_n| = 2^{O(nf(n))}$ admits an adjacency labeling scheme of size  $O( f(n) \log n)$.} 

\begin{proposition}\label{lem:monotone-factorial}
	\monofactorial
\end{proposition}

This upper bound is an easy consequence of an estimation of the number of edges in graphs from monotone classes combined with a standard labeling scheme for $k$-degenerate graphs \cite{KNR88}. 
 Our second main result shows that this upper bound is attained by some monotone classes.
 Before stating the result formally we must briefly introduce a family of non-decreasing functions we call ``decent''.  Roughly speaking, on some domain $[s,\infty)$, decent  functions are sub-multiplicative, i.e.,\ $f(xy) \leq f(x)f(y)$, and moderate-growing, that is $\log x \leq f(x) \leq x^{1-\delta}$ for some constant $\delta\in(0,1)$, see \cref{def:beaut} for the formal definition of decent functions.

 \def\lowerbound{
 	Let $f : \Rrr \rightarrow \Rrr$ be a \dec function. 
 	Then, there exists a monotone graph class $\mathcal{X}$ with speed $|\Xc_n| = 2^{O(nf(n))}$ 
 	that does not admit a universal graph of size at most $2^{f(n) \log n}$. 
 	Equivalently, $\Xc$ admits no adjacency labeling scheme of size at most $ f(n) \log n$.}
 \begin{theorem}\label{th:main}
 	\lowerbound
 \end{theorem}

 	Theorem \ref{th:main} gives the existence of monotone classes requiring labels whose size is a $\log n$-factor above the information-theoretic lower bound. In particular, this shows that \cref{lem:monotone-factorial} is tight. 
A special case of \cref{th:main} (when $f(x) = \log x$) implies that the Implicit Graph Conjecture does not hold even for monotone graph classes. Combining this observation with \cref{lem:monotone-factorial} gives the following result. 

\begin{corollary}\label{cor:factorial}
	For any constant $c>0$, there are factorial monotone classes that do not admit a $(c \log^2 n)$-bit labeling scheme, while any factorial monotone class admits an $O(\log^2 n)$-bit labeling scheme. 
\end{corollary}

This result (more generally Theorem \ref{th:main} and Proposition \ref{lem:monotone-factorial}) gives the first example of tight bounds for families of graph classes that do not admit information-theoretic order optimal adjacency labeling schemes.
Chandoo \cite{Chandoo23} observed that the proof of the refutation of the IGC by Hatami and Hatami \cite{HH22} implies that the family of factorial classes cannot be ``described'' by a~countable set of factorial classes. Using the same ideas, we establish the following result from our proof for monotone classes.  

\def\complex{Let $f:\mathbb{R}_{\geq 0}\rightarrow \mathbb{R}_{\geq 0}$ be any decent function, and $\mathbb{X}$ be any countable set of graph classes, each with speed at most $2^{nf(n)\log n}$. Then, there exists a monotone graph class $\mathcal{X}$ of speed $2^{O(nf(n))}$ such that there does not exist a $\mathcal{D} \in \mathbb{X}$ with $\mathcal{X}\subseteq  \mathcal{D}$.}
\begin{theorem}\label{Thm:complex} 	
	\complex
\end{theorem}
This shows that monotone classes are complex in the sense that they cannot be covered by a~countably infinite family of classes growing slightly faster, even if these classes are not restricted to being hereditary (thus, also to being monotone).

	\subsection{Proof outline and techniques}\label{sec:techneques}
	\label{sec:proof-outline}

	\paragraph{Monotone small classes have bounded degeneracy and implicit representations.}

        We establish \cref{th:smallDegenerate} in the contrapositive: if a monotone class $\Xc$ has unbounded degeneracy, then it is not small.
        To prove this we establish the following two intermediate steps: 
\begin{enumerate}
\item  We first show that every graph of minimum degree~$d$ admits an induced subgraph with minimum~degree at~least~$d$ that has a~spanning tree of \emph{maximum} degree at~most~$d$.  
  
\item Next, we show that if $\Gc = \mon(\{G\})$, where $G$ is any graph with minimum degree $d \geq 1000$, then there exists a $k \in \mathbb{N}$ such that $|\Gc_k| \geq k! \cdot 2^{kd/3}.$
  
          To achieve this, we start from an induced graph $H$ of~$G$ satisfying the previous item, with $k = |V(H)|$ and $m = |E(H)|$.
          Graph $H$ can be shown to have at~least $2^{4m/5}$ pairwise non-isomorphic spanning subgraphs, due to its large minimum degree.
          Let us denote by $\Fc$ this set of subgraphs.
          Crucially each member of $\Fc$ has at~most $2^{m/10}$ automorphisms, due to the spanning tree of bounded maximum degree.
          We conclude since $|\Gc_k| \geq \sum_{F \in \Fc} \frac{k!}{\aut(F)}$.   
\end{enumerate}
Finally, to show the contrapositive of \cref{th:smallDegenerate}, we consider an arbitrary monotone class $\mathcal{X}$ of unbounded degeneracy and assume that for some constant $c$ we have $|\Xc_n| \leq n!\,c^n$ for every $n \in \mathbb{N}$.
Since $\Xc$ has unbounded degeneracy, it contains a graph $G$ with arbitrarily large minimum degree~$d$.
If we take $d$ suitably large, then applying Step 2 to such a graph yields a~contradiction with the assumption of smallness of $\Xc$.

Having established \cref{th:smallDegenerate}, any small monotone class $\mathcal{X}$ has bounded degeneracy. Thus, \cref{cor:monSmall-ir} follows by applying a classical $O(\log n)$-bit labeling scheme for classes of bounded degeneracy \cite{KNR92}, see \cref{lem:degenlabel} for a description of this scheme.  
	\paragraph{Monotone classes that do not admit implicit representations.}
Recall that, roughly speaking\footnote{The formal definition of decent (\cref{def:beaut}) is more general and depends on three parameters $\delta, C,s$. For this proof sketch it suffices to work with the simplified (informal) definition above which only has one parameter $\delta$.}, a function $f : \Rrr \rightarrow \Rrr$ is decent if $\log x \leq f(x) \leq x^{1-\delta}$ for some constant $\delta\in(0,1)$, and $f$ is \Smull, i.e., $f(xy)\leq f(x)\cdot f(y)$, for all $x,y$ in the domain. 	Our approach is inspired by the refutation of the IGC by Hatami and Hatami \cite{HH22}. Namely, for any decent function  $f$, we expose so many \emph{monotone} classes of speed $2^{n f(n)}$ that there are not enough universal graphs of size $2^{f(n) \log n}$ to capture all of them.
	The approach involves several key ingredients:
	\begin{enumerate}
		\item Estimation of the number of sets of graphs of fixed cardinality representable by universal graphs. A~set of graphs $\Mc$ is \emph{representable} by a~universal graph $U$, if every graph in $\Mc$ is an induced subgraph of $U$.
		A~direct estimation shows that the number of sets of cardinality $k_n := \lceil 2^{\sqrt{nf(n)}} \rceil $ of $n$-vertex graphs that are representable by a~$u_n$-vertex universal graph, with $u_n := 2^{ f(n)\log n}$, is at~most 
		\begin{equation}\label{eq:num-representable}
			2^{u_n^2} \cdot u_n^{nk_n} = 2^{2^{2 f(n) \log n} + k_n \cdot n f(n)\log n}.
		\end{equation}
		
		\item Notion of \emph{$f$-good graphs}. 
		We will construct our monotone classes of speed $2^{nf(n)}$ by taking the monotone closure of an appropriately chosen set of graphs.
		The monotonicity and the speed of target classes impose a~natural restriction on the number of edges in graphs that can be used in such constructions.
		To explain, let $\Xc$ be a~monotone class with $|\Xc_n| \leq 2^{nf(n)}$.
		Since $\Xc$ is closed under taking subgraphs,
		if $\Xc$ contains an $n$-vertex graph with $m$ edges, then $\Xc$ contains at least $2^m$ labeled $n$-vertex graphs. This, together with the speed assumption, imply that for any $G \in \Xc$ and $k$, every subgraph of $G$ on $k$ vertices contains at most $kf(k)$ edges.
		
		This restriction, however, is not strong enough for our purposes.
		Indeed, while each graph with the above property contributes to the monotone closure an appropriate number of subgraphs at every \emph{level} (i.e., on every number of vertices), we build our desired classes by taking the monotone closure of \emph{infinitely} many of such graphs, and this can result in some levels having too many graphs.
		To overcome this difficulty, we introduce the notion of \emph{$f$-good} graphs, which are  $n$-vertex graphs in which the number of edges in every $k$-vertex subgraph is at most $kf(k)$
		if $k > \sqrt{n}$, and at most $\frac{kf(k)}{\log k}$ if $2 \leq k \leq \sqrt{n}$. The latter condition ensures that if we take the monotone closure of a set of $f$-good graphs, then all sufficiently small subgraphs of any graph in this class belong to a \emph{fixed} monotone class of speed $2^{nf(n)}$. Namely, the class of all $n$-vertex graphs in which very $k$-vertex subgraph has at most $\frac{kf(k)}{\log k}$ edges for every $2 \leq k \leq n$.
		
		\item Construction of monotone classes of speed $2^{nf(n)}$ from sets of $f$-good graphs.
		We show that for any sequence $(\Mc_{n})_{n \in \Nn}$, where $\Mc_{n}$ is a~set of $f$-good $n$-vertex graphs of cardinality~$k_{n}$, the monotone closure $\mon(\cup_{n \in \Nn} \Mc_{n})$ has speed at most $2^{nf(n)}$.
		
		\item Lower bound on the number of sets of cardinality $k_n$ of $f$-good $n$-vertex graphs.
		We show that for any $\gamma>1$, there exists some $c:=c(\gamma,\delta)>0$ such that for every $n \in \Nn$ there are at least  $2^{(\gamma \delta/2-o(1))\cdot nf(n)\log n}$ many unlabeled $cf$-good $n$-vertex graphs. Thus, the number of sets of cardinality $k_n$ of $cf$-good $n$-vertex graphs is at least 
		\begin{equation}\label{eq:num-good-sets}
			2^{k_n \cdot (\gamma \delta/2-o(1))\cdot n f(n) \log n}.
		\end{equation}
		\label{itme:many-sets}
	\end{enumerate}
	
	By setting $\gamma = 4/\delta$ and recalling that $k_n = \lceil 2^{\sqrt{nf(n)}} \rceil$, 
	we show that \cref{eq:num-good-sets} is larger than \cref{eq:num-representable}. 
	Therefore, there exists a~monotone class $\mon(\cup_{n \in \Nn} \Mc_{n})$ of speed $2^{nf(n)}$ that is not representable by a~universal graph of size~$2^{f(n)\log n}$.
	
	\paragraph{Many~$f$-good graphs.}
	A core step in the above approach is to show that for any $\gamma>1$, there exists some $c:=c(\gamma,\delta)>0$ such that the number of $n$-vertex $cf$-good graphs grows as $2^{(\gamma \delta/2-o(1))\cdot nf(n)\log n}$. To do so, we show that a~random graph $G_n \sim G(n, \gamma f(n)/n)$ is $cf$-good with high probability (\emph{w.h.p.}). It is in this step that we really need to use the \Smul property of decent functions, as we need to relate the magnitude of $f$ at two different parts of its domain. 
	
In particular, to show that w.h.p.\ $G_n$ is $cf$-good, we apply a first moment bound to show there are no ``large'' $k$-vertex subgraphs of $G_n$ with more than $ckf(k)$ edges, and ``small'' ones with more than $ckf(k)/\log k$ edges. Observe that the number of edges $\xi$ in a given $k$-vertex subgraph has expectation $\binom{k}{2}\frac{\gamma f(n)}{n}$. Thus, for ``large'' subgraphs, the probability that $\xi$ is constant factor larger than $ckf(k)$ decays with exponent  $\propto - f(k)\cdot\ln \frac{nf(k)}{kf(n)} $ by the Chernoff bound. From this we see that unless $ f(k)/ f(n)  > k/ n$, then the bound fails. Sub-multiplicativity helps us here as it allows us to say $f(n) = f(k\cdot (n/k))\leq f(k)\cdot f(n/k) $, \Slin then bounds the term $f(n/k)$. A similar issue occurs for ``small'' subgraphs.
	
	From the explanation above it may seem that needing such tight control over the ratio of $f(k)$ to $f(n)$ for all $k\leq n$ is an artifact of our proof, however some ``smoothness'' condition on the function is necessary.   To see this, consider
	a function $f : \Nn \rightarrow \Rrr$ such that $f(n) = \log n$, if $n$ is odd, and $f(n) = \sqrt{n}$, if $n$ is even. Then, for any $c > 0$, and large enough even $n$, $G(n,f(n)/n)$ will not be $cf$-good as the restriction on the subgraphs with odd number of vertices is far too stringent. Sub-multiplicativity was the most natural and broad condition we could find to combat this issue, and we show in \cref{lem:manydecent} that many common functions growing at a suitable rate satisfy this. 
	
	It would be interesting to see if \Smul can be replaced with something more general. We also used \Smul in Step 3 above (which corresponds to \cref{lem:good-tiny}) to bound the speed of $\mon(\cup_{n \in \Nn} \Mc_{n})$, however it is possible that some less stringent property can be used there.

	\paragraph{A matching upper bound on the size of adjacency labels.} 
	We show that for any non-decreasing function $f : \Rrr \rightarrow \Rrr$, any monotone class
	with speed $2^{O(n f(n))}$ admits an $O(f(n)\log n)$-bit labeling scheme.
	This follows from an easy observation that any such class is $O(f(n))$-degenerate, followed by the same standard $O(k \log n)$-bit labeling scheme for $k$-degenerate graphs used to prove \cref{cor:monSmall-ir}.
	One consequence of this upper bound is that our result on the `$f$-goodness' of a random graph (\cref{th:Gnp-sparsness}) is tight: for any $p \in \omega(f(n)/n)$ and $c \geq 0$, a~random graph $G_n \sim G(n, p)$ is not $cf$-good w.h.p.

	\subsection{Discussion}\label{sec:discussion}
	A natural question arising from our work is to characterize monotone classes that admit an implicit representation.
        Motivated by the Implicit Graph Conjecture, of particular interest is the case of factorial classes.
	
	\begin{question}\label{q:monotone-igc}
		Which \emph{monotone} factorial graph classes admit an $O(\log n)$-bit labeling scheme?
	\end{question}
	
	\noindent
	An analogous question is completely understood for constant-size \emph{adjacency sketches} (a probabilistic version of adjacency labeling schemes) that were studied in \cite{FK09,Harms20,HWZ22}. The importance of constant-size adjacency sketches is that they can be derandomized to $O(\log n)$-bit adjacency labels \cite{Harms20,HWZ22}.
	Thus, if a class admits constant-size adjacency sketches, then
	it admits an $O(\log n)$-bit labeling scheme. Though, the converse is not always true. Esperet, Harms, and Kupavskii showed \cite{EHK22} that a monotone class admits constant-size adjacency sketches if and only if it has bounded degeneracy.
	This result may suggest that bounded degeneracy also characterizes monotone classes that admit $O(\log n)$-bit labeling schemes. This, however, is not the case, as the class of subgraphs of hypercubes is monotone, has unbounded degeneracy, and admits an $O(\log n)$-bit labeling scheme \cite{EHZ23}.
	
Recall that \cref{q:gen-igq} (first raised in \cite{Spinrad03}), asks which hereditary graph classes admit implicit representations.
	A prominent instance of \cref{q:gen-igq} is whether every hereditary small class admits an implicit representation. It was shown in \cite{BDSZZ23} that for any $\kappa>0$ there is a monotone small class which does not admit a ($\kappa \log n$)-bit labeling scheme; in particular, some monotone small classes admit no \emph{information-theoretic asymptotically optimal} labeling scheme. One of our main results (\cref{cor:monSmall-ir}) shows that every monotone class admits an \emph{information-theoretic order optimal} labeling scheme, i.e., an implicit representation. We conjecture that the same holds for all hereditary small classes.
	\begin{conjecture}[Small Implicit Graph Conjecture]\label{conj:sigc}
		Any hereditary small class admits an implicit representation. 
	\end{conjecture}
\Cref{conj:sigc} is also known to hold for classes of bounded twin-width~\cite{BGK22}.

 We conclude this discussion with a more technical (yet natural) question of whether the conditions (\Slin and \Smul) of ``decent'' can be relaxed. Due to the discussion under the heading ``Many $f$-good graphs'' in \cref{sec:techneques}, the \Slin condition is essentially necessary, and  if one is to follow our method, some notion of global ``smoothness'' is required to prove \cref{th:Gnp-sparsness}.	
 	However, it is not so clear to what extent the \Smul condition is necessary to achieve the required ``smoothness''.

	\subsection{Organization}	
	The rest of the paper is organized as follows. 
	In Section \ref{sec:basics}, we cover some common notation and definitions. 
	In \cref{sec:phase-transition}, we prove our first main result, namely, that any monotone small class has bounded degeneracy, and therefore admits an implicit representation.
	\cref{sec:ingredients,sec:monotone} are devoted to our second main result, namely, tight bounds on the size of adjacency labeling schemes for monotone classes.
	In \cref{sec:GoodDecent} we introduce two key concepts used in the proofs. Firstly, we give the notion of \emph{$f$-good} graphs, which are the building blocks for the monotone classes used to prove the lower bounds. Secondly, we formally define \emph{decent} functions which describe the speeds of these monotone graph classes, before concluding \cref{sec:GoodDecent} with some natural examples of decent functions. In \cref{sec:Gnp-subgraphs}, we prove a result about random graphs which is the main technical ingredient of our lower bounds.
	In \cref{sec:monotone}, we establish the lower and upper bounds on labeling schemes for monotone classes, along with the result on the complexity of monotone graph classes.

	\section{Standard definitions and notation }\label{sec:basics}
    We let $[n]$  denote the set $\{1, \dots, n\}$ of natural numbers, and use $\ln^c x$ as a~shorthand for $(\ln x)^c$.  We take $\Rrr$ to denote the set of non-negative real numbers.    
    We use $X\sim \mathcal{D}$ to denote that the random variable $X$ has distribution $\mathcal{D}$. We say that a sequence of events $(A_n)$ holds \emph{with high probability (w.h.p.)} if $\Pr{A_n}\rightarrow 1$ as $n\rightarrow \infty$.
        
	\paragraph{Graphs.} We consider finite undirected graphs, without loops or multiple edges.
        Given a~graph~$G$, we write~$V(G)$ for its vertex set, and~$E(G)$ for~its edge set.	
	A graph $H$ is a~\emph{subgraph} of~$G$ if $V(H) \subseteq V(G)$ and $E(H) \subseteq E(G)$.
        Thus, $H$ can be obtained from $G$ by vertex and edge deletions.
	The graph $H$ is an \emph{induced subgraph} of $G$ if $V(H) \subseteq V(G)$,
	and~$E(H)$ consists exactly of the edges in~$E(G)$ with both endpoints in~$V(H)$.
        In that case, $H$ can be obtained from $G$ by vertex deletions only. In the usual way, for a~set of vertices $U\subseteq V(G)$, we denote by $G[U]$ the induced subgraph of $G$ with the set of vertices $U$.
        We denote by $e(G)$ the number of edges in $G$
        
        When we refer to an $n$-vertex graph $G$ as \emph{labeled}, we mean that the vertex set of $G$
        is $[n]$, and we distinguish two different labeled graphs even if they are isomorphic.
        In contrast, if we refer to $G$ as \emph{unlabeled} graph, its vertices are indistinguishable and
        two isomorphic graphs correspond to the same unlabeled graph.

	\paragraph{Graph classes.} A~graph class is \emph{hereditary} if it is closed under taking induced subgraphs, and it is \emph{monotone} if it closed under taking subgraphs.
	For a~set $\Xc$ of graphs we let $\her(\Xc)$ denote the hereditary closure of $\Xc$, i.e., the inclusion-wise minimal hereditary class that contains $\Xc$; and $\mon(\Xc)$ denote the monotone closure of $\Xc$, i.e., the minimal monotone class that contains $\Xc$.

	\section{Monotone small classes admit an implicit representation}
	\label{sec:phase-transition}
	
	In this section we show that any monotone small class admits an implicit representation. 
	To~do~so, we first establish (Theorem \ref{th:smallDegenerate}) that any small monotone classes has bounded degeneracy.
        This result has a broader scope than just implicit representations, and is of independent interest.
        For example, it generalizes the fact that monotone classes of bounded twin-width have bounded degeneracy~\cite[$(iv) \Rightarrow (iii)$ in Theorem 2.12]{BGK22}.
	The labeling scheme for monotone small classes then follows from a classical labeling scheme for graphs of bounded degeneracy, see~\cref{lem:degenlabel}. 
	
	We proceed with some notation and known auxiliary facts that we will employ in the proof. 
	Recall that for a class of graphs $\Xc$, we denote by $\Xc_n$ the set of graphs in $\Xc$ with vertex set $[n]$.
        We will denote by $\Xc_n^u$ the set of \emph{unlabeled} $n$-vertex graphs in $\Xc$, i.e., the set of isomorphism classes in~$\Xc_n$. 
	Observe that for an unlabeled $n$-vertex graph $G$ there are exactly $\frac{n!}{\aut(G)}$ labeled graphs isomorphic to $G$, where $\aut(G)$ is the order of the automorphism group of $G$.
        Thus we have
	\[
	|\Xc_n| = \sum_{G \in \Xc_n^u} \frac{n!}{\aut(G)}.
	\]
	
	Let $F$ be a spanning subgraph of a fixed labeled graph $G$.
        Thus, we recall, $F$ is defined by a~subset of $E(G)$.
	We denote by $\sub(F \rightarrow G)$ the number of subgraphs of $G$ isomorphic to $F$, and by $\emb(F \rightarrow G)$ the number of embeddings of $F$ into $G$, i.e., the number of permutations from~$\mathfrak S_n$ that map $F$ to an isomorphic copy of $F$ in $G$. 
	Thus, \[\emb(F \rightarrow G) = \sub(F \rightarrow G)\cdot\aut(F).\]
	We will use the following well known facts (see e.g. \cite{KLT02}).
	
	\begin{lemma}\label{lem:autGspanF}
		Let $F$ be a spanning subgraph of a graph $G$. Then 
		\[
			\aut(G) \leq \emb(F \rightarrow G) = \sub(F \rightarrow G)\cdot\aut(F).
		\]
	\end{lemma}
	
	\begin{lemma}\label{lem:autGdeg}
		Let $G$ be a connected graph of maximum degree $\Delta$.
		Then 
		\[
		\aut(G) \leq n \cdot \Delta! \cdot (\Delta-1)^{n-\Delta-1} \leq n \Delta^{n-1}.
		\]
	\end{lemma}

	We start with the following deceptively simple lemma.
        	
	\begin{lemma}\label{lem:minDegSubgraphs}
	  Let $G$ be a graph of minimum degree $d$.
          Then $G$ has an induced subgraph $H$ of minimum degree at~least~$d$ with a~spanning tree of \emph{maximum degree} at~most~$d$.
	\end{lemma}
	\begin{proof}
		Let $T$ be an inclusion-wise maximal tree among subgraphs of $G$ with maximum degree at~most~$d$.
                Let $H$ be $G[V(T)]$.		
		Note that any vertex $v$ of $H$ that has degree less than $d$ in $T$, has no neighbors among $V(G) \setminus V(H)$ in $G$, as otherwise $T$ could have been extended by	adding any neighbor of $v$ from $V(G) \setminus V(H)$, which would contradict the maximality of $T$. 
		Thus all neighbors of $v$ are in $V(H)$, which implies that the minimum degree of $H$ is at~least $d$. 
	\end{proof}

         We show next that if a~graph has large minimum degree and a~spanning tree of bounded maximum degree, then it contains many pairwise non-isomorphic spanning connected subgraphs with a~small number of automorphisms.
	
	\begin{lemma}\label{lem:minDegSubgraphsSpanTree}
	  Let $G$ be an $n$-vertex $m$-edge connected graph of minimum degree $d \geq 1000$, with a~spanning tree of maximum degree at most $d$.
          Then, $G$ contains
		at least $2^{4m/5}$ pairwise non-isomorphic spanning connected subgraphs $F$
		with $\aut(F) \leq 2^{m/10}$. Consequently, $\mon(\{G\})$ contains at~least $n! \cdot 2^{nd/3}$ graphs on vertex set $[n]$.
	\end{lemma}
	
	\begin{proof}
		Fix a spanning tree $T$ of $G$ of maximum degree at most $d$.
		Denote by $\Fc$ the family of all subgraphs of $G$ containing $T$. Then, 
		\[
		|\Fc| = 2^{m-n+1} \geq 2^{m - \frac{2m}{d}} \geq  2^{9m/10},
		\]
		where we used the fact that $n \leq 2m/d$ by the assumption on the minimum degree of $G$, and the assumption $d\geq 1000$.

		For a fixed graph $F \in \Fc$, we will now estimate the number $N_{\Fc}(F)$ of graphs in $\Fc$ that are isomorphic to $F$. This number is at most the number $\emb(T \rightarrow G) = \sub(T \rightarrow G)\cdot \aut(T)$ of embeddings of $T$ into $G$. The number $\sub(T \rightarrow G)$ of subgraphs of $G$ isomorphic to~$T$ is at most 
		\begin{equation*}
			\binom{m}{n-1} \leq \binom{m}{2m/d} \leq \left(\frac{ed}{2} \right)^{2m/d}\leq   2^{m/20},
		\end{equation*}
		where in the last inequality we used the assumption $d \geq 1000$.
		Recalling that the maximum degree of~$T$ is at most $d$ and using \cref{lem:autGdeg}, we conclude that $\aut(T) \leq n d^{n-1} \leq 2^{2n \log d}$.
		Consequently,
                \begin{equation}\label{eq:TtoG} 
			N_{\Fc}(F) \leq \emb(T \rightarrow G) \leq 2^{m/20}\cdot 2^{ 2n \log d} \leq 2^{m/20 + 2(2m/d)\log d } \leq 2^{m/10}, 
		\end{equation} where again we used $n\leq 2m/d$ and $d\geq 1000$. Note that the bound in \eqref{eq:TtoG} holds for any $F\in \mathcal{F}$. Thus, the number of pairwise non-isomorphic graphs in $\Fc$ is at least \[\frac{|\Fc|}{\max_{F\in \mathcal{F}}N_{\Fc}(F)} \geq 2^{9m/10}\cdot 2^{- m/10} = 2^{4m/5}.\]
		
		Furthermore, for any $F\in \mathcal{F}$, we have 
		\[
		\aut(F) \leq \sub(T \rightarrow F)\cdot\aut(T) \leq \sub(T \rightarrow G)\cdot\aut(T) = \emb(T \rightarrow G) \leq 2^{m/10}, 
		\]
		where we used \cref{lem:autGspanF}, the fact that $\sub(T \rightarrow F) \leq \sub(T \rightarrow G)$, and \eqref{eq:TtoG}.

         Finally, the number of graphs with vertex set $[n]$ isomorphic to a~subgraph of $G$ is at~least
		\[
		\sum_{F \in \Fc} \frac{n!}{\aut(F)} 
		\geq 2^{4m/5} \cdot \frac{n!}{2^{m/10}} 
		= n! \cdot 2^{7m/10} 
		\geq  n! \cdot 2^{dn/3},
                \]
                where in the last inequality we used $m\geq dn/2$. 
        \end{proof}

        We can now prove the main result of this section.
        
        {\renewcommand{\thetheorem}{\ref{th:smallDegenerate}}
		\begin{theorem}
			\smallDgn	
	\end{theorem}}
	\addtocounter{theorem}{-1}
	\begin{proof}
		To prove the theorem we will show the contrapositive, i.e., if a class $\Xc$ has unbounded degeneracy, then it is not small.
		Suppose towards a contradiction that $\Xc$ has unbounded degeneracy,
		but there exists a constant $c$, such that $|\Xc_n| \leq n! \cdot c^n$ holds for every $n \in \mathbb{N}$.
		
		Since $\Xc$ has unbounded degeneracy and is closed under taking subgraphs, for any positive $d$, the class $\Xc$ contains a~connected graph with minimum degree at least $d$.
        Fix any $d > \max\{ 1000, 3\log c \}$, and let $G \in \Xc$ be a connected graph of minimum degree at least $d$.
                Let $\Gc = \mon(\{G\})$. 
		By \cref{lem:minDegSubgraphs,lem:minDegSubgraphsSpanTree}, for some $n \in \mathbb{N}$ we have 
		\[
		|\Gc_n| \geq n! \cdot 2^{nd/3} > n! \cdot 2^{n\log c} = n! \cdot c^n,
		\]
		where the first strict inequality holds due to $d > 3 \log c$.
		Since $\Gc \subseteq \Xc$, this is in contradiction with the assumed upper bound on the number of labeled graphs in~$\Xc$.
		Thus $\Xc$ is not small.
	\end{proof}
		
        The previous result is of independent interest, and provides some structural insight on monotone small classes.
        To show that some property generally holds on small monotone classes, one can now use their degeneracy.
        We give the first such application of~\cref{th:smallDegenerate}. 	
	
	{\renewcommand{\thetheorem}{\ref{cor:monSmall-ir}}
		\begin{theorem}
			\smallcor	
		\end{theorem}}
		\addtocounter{theorem}{-1}

The relevance of \cref{th:smallDegenerate} to labeling schemes should be clear from the following folklore bound~\cite{KNR92}, which we recall for completeness.
	
	\begin{lemma}\label{lem:degenlabel}
		The class of $k$-degenerate graphs has a  $(k+1)\cdot\lceil\log n\rceil $-bit adjacency labeling scheme. 
	\end{lemma} 
	\begin{proof}
		For any $k$-degenerate graph $G$ on $n$-vertices, we first order vertices so that each vertex has at most $k$ neighbors appearing after it in the ordering. This can be done greedily since each subgraph has a vertex of degree at most $k$. One can then assign each vertex a label consisting of its place in the order, followed by the places of the at most $k$ neighbor vertices following it in the ordering.  
	\end{proof}
	
\cref{cor:monSmall-ir} follows directly from \cref{lem:degenlabel} and \cref{th:smallDegenerate}.

\section{Ingredients for the proof of the lower bound}
\label{sec:ingredients}
		
	This section contains many of the components needed to construct the classes used in the proof of the lower bound (\cref{th:main}). In \cref{sec:GoodDecent} we introduce several notions related to subgraph density, which are then applied to random graphs in \cref{sec:Gnp-subgraphs}. 
		\subsection{Good graphs and decent functions}\label{sec:GoodDecent}

		Our first definition describes graphs which do not have overly dense subgraphs. 
		\begin{definition}[$f$-good]\label{def:good}
			Let $f: \Rrr \rightarrow \Rrr$ be a function.
			An $n$-vertex graph $G$ is \emph{$f$-good} if the number of edges in any subgraph on $k$ vertices is bounded from above by 	
			\[\begin{cases}
				\frac{k\cdot f(k)}{\log k} & \text{ if }  2 \leq k \leq  \sqrt{n}   \\
				k\cdot f(k) & \text{ if }   \sqrt{n} < k \leq n  
			\end{cases}.\]  
		\end{definition}
		
		We observe that $f$-goodness is a monotone property, i.e.,\ if a graph $G$ is $f$-good, then so is any of its subgraphs.
		Indeed, moving the threshold (between the first and the second, more relaxed, upper bound) from $\sqrt n$ down to a~smaller value may only help in satisfying these bounds.
		
		The next definition gives a class of functions used to describe speeds of monotone classes, where, for such a function $f(n)$, we will consider classes of growth $2^{nf(n)}$. 
		\begin{definition}[$(\delta,C,s)$-\emph{\dec}]\label{def:beaut} For constants $\delta \in (0,1)$, $C\geq 1$ and $s\geq 2$, we say that a non-decreasing function $f :\Rrr  \rightarrow \Rrr$ is $(\delta,C,s)$-\emph{\dec} if the following properties hold
			\begin{description}
				\item[(Moderate-growth):] $\log x \leq f(x) \leq C\cdot x^{1 - \delta}$ holds for every $x \in [s,\infty)$, 
				\item[(Sub-multiplicativity):]  $f(xy) \leq C\cdot f(x)\cdot f(y)$ holds for any $x,y \in [s,\infty)$.
			\end{description}
		\end{definition}

		
		We say that a function $f $ is \textit{\dec} if there  exist some constants $\delta\in (0,1)$, $C\geq 1 $, and $s\geq 2$ such that $f$ is $(\delta,C,s)$-\dec. For any constant $C\geq 1$, the function $f(x)=C\log x$ is decent; this captures factorial growth. We now give some other natural examples of \dec functions. 
		
		\begin{lemma}\label{lem:manydecent}
			For any fixed $\alpha>0,\beta\geq 1 ,\gamma \geq 1 $ and $d\in (0,1)$, the following functions are \dec:
			\begin{enumerate}[(i)]
				\item  $f(x) = \alpha x^d$,
				\item $f(x) = \exp\left(\alpha \cdot  \ln^d x \right) $,
				\item $f(x) = \exp\left(\beta \cdot  \ln^\gamma(\log x) \right) $,
				\item $f(x) = \beta \cdot g(x)$, where $g(x)$ is \dec,
				\item $f(x) = g(x)\cdot h(x)$, where $g(x),h(x)$ are \dec and  $g(x)\cdot h(x)$ is \Slinn. 	
			\end{enumerate}
		\end{lemma}
		\begin{proof}For $(i)$, if we set $s:= \left(\frac{2}{d\max\{\alpha, 1\}}\right)^{2/d}$, then we have \[f(s) = \alpha \cdot \left(\frac{2}{d\max\{\alpha, 1\}}\right)^{2} \geq \frac{2}{d}\cdot  \frac{2}{d\max\{\alpha, 1\}} \geq \frac{2}{d} \cdot \log \frac{2}{d\max\{\alpha, 1\}}  = \log s.  \] Furthermore, observe that $\frac{\alpha x^d}{\log x}$ is increasing for all $x> e^{1/d}$, and that $s>2^{2/d}>e^{1/d}$. Thus, $f$ has \Slin on $[s, \infty)$ with $C=\alpha$ and $\delta =1-d$. Observe also that $f(xy) = \frac{1}{\alpha}\cdot f(x)f(y)$, and thus  $f(x)=\alpha x^d $ is $\Big(1-d,\max\left\{\alpha ,\frac{1}{\alpha}\right\}, s\Big)$-\dec. 
			
			For $(ii)$, \Slin holds for $C=1$, any fixed $\delta \in(0,1)$, and sufficiently large $s\geq 2$. For $x,y\geq 0$ let $g_x(y) = (x+y)^d - x^d-y^d$ and observe that $g_x(0)=0$ and $g_x'(y) = d(x+y)^{d-1} - dy^{d-1}\leq 0 $.  Consequently, $g_x(y) \leq 0$ for all $x,y\geq 0$, or equivalently $(x+y)^d \leq x^d+y^d $. This implies that $f$ is \Smull as  
			\[f(xy) = \exp\left(\alpha (\ln x + \ln y)^d\right)\leq \exp\left(\alpha ((\ln x)^d + (\ln y)^d)\right) = f(x)\cdot f(y).\]

			For $(iii)$,  it will be useful to show that   \begin{equation}\label{eq:logsubadd} 
				\ln^\gamma(x+y)  \leq \ln^\gamma x  + \ln^\gamma y, \qquad \text{for all }x,y \in [e^{\gamma},\infty). 
			\end{equation}

			To prove \eqref{eq:logsubadd}, we first observe that the function $g(x)= \frac{\ln^\gamma x}{x}$ is non-increasing for $x\in [e^\gamma,\infty)$. This follows since  $g$ is differentiable when $x\neq 0$ and $g'(x) = \frac{(\gamma -\ln x )\ln^{\gamma -1 } x}{x^2} < 0 $ for all $x> e^\gamma>1$. Thus \eqref{eq:logsubadd} follows from this observation, since for any $x,y \in [e^\gamma,\infty)$ we have 
			\[\ln^\gamma(x+y) = x\cdot \frac{\ln^\gamma(x+y) }{x+y} + y\cdot \frac{\ln^\gamma(x+y) }{x+y} \leq  x\cdot \frac{\ln^\gamma x }{x} + y\cdot \frac{\ln^\gamma y}{y}= \ln^\gamma x + \ln^\gamma y. \]   
			We now see that $f$ is \Smull for any $x,y \in [2^{e^\gamma},\infty)$ as by \eqref{eq:logsubadd} we have 
			\[f(xy) =  \exp\left(\beta\cdot \ln^\gamma(\log x +\log y) \right)\leq \exp\left(\beta \cdot (\ln^\gamma(\log x)+ \ln^\gamma(\log y) ) \right)=f(x)\cdot f(y).\]   Since $\gamma \geq 1$ and $\beta\geq 1$, $f$ is also \Slinn for a sufficiently large $s$.

			For $(iv)$, if $g$ is $(\delta_g,C_g,s_g)$-\dec, then it is easy to check that $\beta g$ is $(\delta_g,\beta C_g,s_g)$-\dec.

			For $(v)$, let $g$ be $(\delta_g,C_g,s_g)$-\dec and $h$ be $(\delta_h,C_h,s_h)$-\dec, and $f(x):=g(x)\cdot h(x)$. As $\log x \leq f(x) \leq C'x^{1-\delta'}$ for some $\delta' \in (0,1),C'>0 $, and $s'\geq 2$, by assumption, it remains to show \Smul. For any $x,y \in [\max\{s_g,s_h\}, \infty)$  we have
			\[f(xy)=g(xy)\cdot h(xy)\leq C_g g(x)g(y)\cdot C_h h(x)h(y) \leq C_g\cdot C_h\cdot  f(x)f(y),\]  and thus $f$ is $(\delta', \max\{C',C_g\cdot C_h\},\max\{s',s_g,s_h\})$-\dec.
		\end{proof}

	\subsection{Growth of the number of edges in subgraphs of random graphs}
	\label{sec:Gnp-subgraphs}
 The aim of this Section is to show that there are many graphs which are suitable for building the classes we need to prove \cref{th:main}.  We will achieve this using random graphs, where $G(n,p)$ denotes the distribution on $n$-vertex graphs where each edge is included independently with probability $p$, see (for example) \cite{FriezeKaronski}. Our main result shows that random graphs are suitable with high probably.

	\begin{theorem}\label{th:Gnp-sparsness}
	Let $f : \Rrr \rightarrow \Rrr$ be $(\delta,C,s)$-\dec  for some constants $\delta \in (0,1)$, $C\geq 1$, and $s\geq 2$. Then, for any fixed $ \gamma > 1$, there exists $c:= c(\delta,C,s,\gamma)>0$ such that, for large $n$,
		\[
			\mathbb{P}\left[\, G(n, \gamma  f(n)/n)\text{ is not }(c f)\text{-good } \right] \leq n^{-2}.
		\]

	\end{theorem}

To prove this we will utilize the following version of the Chernoff bound (see \cite[Theorem A.1.15]{ProbMethod}), where $\mathrm{Bin}(N,p)$ denotes the binomial distribution with parameters $N$ and $p$. 
\begin{lemma}[Chernoff bound]\label{lem:Chernoff}
	Let $\xi\sim\mathrm{Bin}(N,p)$, $\mu=Np$, and  $a,t>0$. Then, 
	\begin{equation*}  \mathbb{P}(\xi>(1+a)\mu)\leq\left(\frac{e^a}{(1+a)^{1+a}}\right)^{\mu}\leq  \exp \left(-(1+a)\mu\cdot \ln  \frac{1+a}{e}\right).
	\end{equation*} 
\end{lemma}

	\begin{proof}[Proof of \cref{th:Gnp-sparsness}] 	Let $p := p(n) = \gamma  f(n)/n$, and let $c_1,c_2$ be sufficiently large constants (depending on $\gamma$) fixed later. Let $\mathcal{E}_{1,k}$ (respectively $\mathcal{E}_{2,k}$) be the event that there are no subgraphs of size $k$ with more than $c_1 kf(k)/\log k$ edges (respectively $c_2 kf(k)$ edges). Observe that if $c=\max\{c_1, c_2, \binom{s}{2} \}$, then 
		\begin{equation}\label{eq:events} \left\{ G(n, p)\text{ is not } (c f)\text{-good}  \right\}\subseteq \left(\bigcup_{k= s}^{\lfloor \sqrt n \rfloor  } \neg \mathcal{E}_{1,k}\right) \cup \left(\bigcup_{k= \lfloor \sqrt n \rfloor +1   }^n \neg \mathcal{E}_{2,k}\right). \end{equation}

	Let  $k$ denote the number of vertices in a subgraph, and thus $\xi\sim\mathrm{Bin}\left({k\choose 2},p\right)$ denotes the number of edges in a given $k$-vertex subgraph.  The expectation of $\xi$ is  
	\begin{equation*} 
		\mu:={k\choose 2}p = \frac{\gamma }{2}\cdot \frac{k(k-1)f(n)}{n}.
	\end{equation*} 
	On the other hand, the number of ways to select a $k$-vertex subgraph is  
\begin{equation}\label{eq:nchoosek}
	{n\choose k}\leq\left(\frac{en}{k}\right)^k=\exp\left(k\ln\frac{n}{k}+k\right) \leq \exp\left(2k\ln n\right).\end{equation}
Our strategy will be to bound the probability of the events on the right-hand side of \eqref{eq:events} using the union and Chernoff bounds.

We begin by considering events of the form $\mathcal{E}_{2,k}$ and thus can assume that $\lfloor \sqrt{n}\rfloor +1 \leq k\leq n  $. Observe that since $f$ is \Smull, non-decreasing,  and \Slinn, we have 
\begin{equation}\label{eq:subaddf}
	\frac{f(k)}{ f(n)}  =  \frac{f(k)}{ f\left(\frac{n}{k}\cdot  k\right)}  \geq   \frac{f(k)}{ C\cdot f( \frac{n}{k})\cdot f(k)}  \geq   \frac{f(k)}{ C\cdot f( \frac{sn}{k})\cdot f(k)}  \geq \frac{f(k)}{ C^2\cdot ( \frac{sn}{k})^{1-\delta} \cdot f(k)} \geq \frac{k}{C^2s\cdot n}  
	.\end{equation}If we now fix \begin{equation}\label{eq:consstc2} c_2  = C^2s\cdot  e^2 \cdot \gamma >6 ,\end{equation} then by \eqref{eq:subaddf} we have 
\begin{equation}\label{eq:keyGnp} \frac{2c_2nf(k)}{e\gamma(k-1)f(n)} =  \frac{ 2C^2 s e  \cdot nf(k)}{(k-1)f(n)} \geq  \frac{ 2e k}{k-1} > e.   \end{equation}
So, applying Chernoff bound (\cref{lem:Chernoff})  with $1+a= \frac{c_2 k f(k)}{\mu} = \frac{2c_2nf(k)}{\gamma(k-1)f(n)} $ gives
\begin{align}\label{eq:chb}
\mathbb{P}(\xi>c_2 kf(k))&\leq\exp \left(-(1+a)\mu\cdot \ln  \frac{1+a}{e}\right) \notag \\
&=  \exp\left(-c_2 k f(k)\cdot\ln \frac{2c_2nf(k)}{e\gamma(k-1)f(n)}  \right) \notag \\
&\stackrel{\eqref{eq:keyGnp}}{\leq} \exp\left(-c_2 k f(k)   \right)\notag\\
&\stackrel{\eqref{eq:consstc2}}{\leq} \exp\left(-6 k f(k)   \right). \end{align}

Thus, by \eqref{eq:nchoosek}, \eqref{eq:chb}, the union bound, and as $f(k) \geq \log k > \ln k$, we have
\begin{equation} \label{eq:union1}
	\mathbb{P}\left(\bigcup_{k = \lfloor \sqrt{n}\rfloor +1}^n\neg \mathcal{E}_{2,k}  \right)  
	\leq \sum_{k = \lfloor \sqrt{n}\rfloor + 1}^n  \exp\left(2k\ln n\right)\cdot  \exp\left(-6k f(k)   \right)\leq \sum_{k = \lfloor \sqrt{n}\rfloor + 1}^n k^{-k} \leq  \exp(- \sqrt{n}) .
\end{equation}

We now treat events of the form $\mathcal{E}_{1,k}$, and thus we can assume that $s\leq k\leq \lfloor \sqrt{n} \rfloor $. 
Observe that for any fixed constant $d>0$ and sufficiently  large $n$ we have $\frac{n^{2/3}}{ k(\log k)^d} \geq s $  as $k\leq \sqrt{n}$.  Thus, by \Smul,  and \Slin we have
\begin{align*}f\left( \frac{n^{2/3}}{(\log k)^d} \right) &= f\left(  \frac{n^{2/3}}{ k(\log k)^d}   \cdot k \right)\\
	& \leq C\cdot  f\left(  \frac{n^{2/3}}{ k(\log k)^d} \right) \cdot f\left( k \right) \\
	 &\leq C^2\cdot    \left(\frac{n^{2/3}}{ k(\log k)^d}\right)^{1-\delta} \cdot f(k)\\
	 & \leq C^2\cdot  \frac{n^{2/3}}{ k(\log k)^d} \cdot f(k). \end{align*}
Similarly, by \Smul and \Slin, we have
\begin{align*}f(n)  &=  f\left(  \frac{n^{2/3}}{(\log k)^d}  \cdot n^{1/3}(\log k)^d  \right)\\ 
	&\leq C\cdot  f\left(  \frac{n^{2/3}}{(\log k)^d} \right) \cdot f\left(n^{1/3}(\log k)^d  \right) \\
	&\leq C^2\cdot  f\left(  \frac{n^{2/3}}{(\log k)^d}  \right)\cdot n^{(1-\delta)/3}(\log k)^{(1-\delta)\cdot d}.\end{align*} If we set $d = 1/\delta >0$ then the two bounds above give  
\begin{equation}\label{eq:subaddff}
	\frac{f(k)}{ f(n)}  
	\geq \frac{f\left(  \frac{n^{2/3}}{(\log k)^d}  \right) \cdot \frac{k (\log k)^d }{C^2n^{2/3}}}{C^2\cdot f\left(  \frac{n^{2/3}}{(\log k)^d}  \right)\cdot n^{(1-\delta)/3}(\log k)^{(1-\delta)\cdot d} } = \frac{k (\log k)^{\delta d} }{C^4n^{1-\delta/3}} 
	= \frac{k \log k }{C^4n}\cdot n^{\delta/3} 
	.\end{equation} Foreseeing the need for the constant $15$ later on, we now set \begin{equation}\label{eq:constc1} c_1  =e\cdot 15 \cdot C^4 \gamma/\delta .\end{equation} We now set $1+a:= \frac{c_1 k f(k)}{\mu\cdot \log k}$, which by \eqref{eq:subaddff}  satisfies
	\begin{equation}\label{eq:secondabdd}1+a= \frac{c_1 k f(k)}{\mu\cdot \log k} = \frac{2c_1nf(k)}{\gamma(k-1)f(n)\log k }\geq \frac{2c_1k }{\gamma (k-1) C^4 }\cdot n^{\delta/3}  > e\cdot n^{\delta/3} . \end{equation}   As before, Chernoff bound (\cref{lem:Chernoff}) with this $1+a$ gives 
\begin{align}\label{eq:chb2}
	\mathbb{P}\left(\xi>\frac{c_1 kf(k)}{\log k}\right)&\leq  \exp\left(-\frac{c_1 kf(k)}{\log k}\cdot\ln \frac{1+a}{e}  \right)\notag\\
	&\stackrel{\eqref{eq:secondabdd}}{\leq}   \exp\left(-\frac{c_1 kf(k)}{\log k}\cdot\frac{\delta }{3}\ln n  \right)\notag\\
		&\stackrel{(*)}{\leq}  \exp\left(- c_1 k \cdot\frac{\delta }{3}\ln n  \right)\notag\\
		&\stackrel{\eqref{eq:constc1}}{\leq}  \exp\left(-5 k  \ln n  \right)
	,\end{align}
where $(*)$ follows as $f(k)\geq \log k$ by \Slin. Thus, by \eqref{eq:nchoosek}, \eqref{eq:chb2}, and the union bound, 
 \begin{equation} \label{eq:union2}
 	\mathbb{P}\left(\bigcup_{k = s}^{\lfloor \sqrt{n}\rfloor } \neg \mathcal{E}_{1,k}  \right) \leq \sum_{k = s}^{\lfloor \sqrt{n}\rfloor }   \exp\left(2k\ln n\right)\cdot  \exp\left(-5 k \ln n   \right) \leq \sqrt{n}\cdot n^{-3s} \leq  n^{-5} .
 \end{equation} The result follows by taking $c = \max\{c_1,c_2, \binom{s}{2}\} $, \eqref{eq:events}, and the union bound over \eqref{eq:union1} and \eqref{eq:union2}. 
   \end{proof}
 We now use \cref{th:Gnp-sparsness} to prove \cref{lem:superfactorial}, which bounds the number of $cf$-good graphs from below. To prove \cref{lem:superfactorial}, it is convenient to switch to an alternative model of random graphs with a fixed number of edges. We let $G(n,m)$ to denote the uniform distribution on $n$-vertex graphs with $m$~edges. The following lemma allows us to transfer results from one graph model to another. 

\begin{lemma}\label{lem:transfer}
	Let $\mathcal{P}$ be any graph property (i.e., graph class) and $0\leq p\leq 1$ satisfy $p\binom{n}{2}\rightarrow \infty$ and $\binom{n}{2} -p\binom{n}{2}\rightarrow \infty$ and $m =\left\lceil p\binom{n}{2} \right\rceil $. Then, for $G_n \sim G(n,m)$ and $G_n' \sim G(n,p)$, we have  
	\[ \Pr{G_n\in \mathcal{P}} \leq 10 \sqrt{m}\cdot  \Pr{G_n' \in \mathcal{P}}. \]	
\end{lemma}

Lemma \ref{lem:transfer} follows by a~very minor adaption of \cite[Lemma 3.2]{FriezeKaronski}, the only difference is a~ceiling in the number of edges, which makes no difference in the proof.

 	\begin{lemma}\label{lem:superfactorial}
				Let $f : \Rrr \rightarrow \Rrr$ be $(\delta,C,s)$-\dec  for some constants $\delta \in (0,1)$, $C\geq 1$, and $s\geq 2$.   
 		Then, for any fixed $\gamma>1$, there exists some $c:=c(\gamma,\delta,C,s)>0$ such that for every $n \in \Nn$ there are at least  
 		$ 2^{(\gamma \delta/2-o(1))\cdot nf(n)\log n}$ many unlabeled $(cf)$-good $n$-vertex graphs.
 	\end{lemma} 
 	\begin{proof}Let $m:= \big\lceil \tfrac{\gamma(n-1)f(n)}{2}\big\rceil$ and $G_n \sim G\big(n,m \big)$. Observe that by \cref{th:Gnp-sparsness} and \cref{lem:transfer}, there exists some fixed $c >0$ such that for sufficiently large $n$   
 		\begin{equation}\label{eq:probgood}
 			\Pr{G_n\text{ is $(cf)$-good}} \geq 1-10 \sqrt{\big\lceil \tfrac{\gamma(n-1)f(n)}{2}\big\rceil}\cdot n^{-2} = 1-o(1).
 		\end{equation}
                The number of labeled graphs in the support of $G\big(n,m\big)$ is \[\binom{\binom{n}{2}}{m}=\binom{\binom{n}{2}}{\left\lceil \tfrac{\gamma (n-1)f(n)}{2}\right\rceil }\geq\left(\frac{n}{\gamma  f(n)}\right)^{\tfrac{\gamma (n-1)f(n)}{2}} =  2^{\frac{\gamma }{2}\cdot (n-1)f(n) \cdot (\log n  - \log (\gamma f(n)))} . \]
                            By \eqref{eq:probgood}, a~$1-o(1)$ fraction of these labeled graphs are $(cf)$-good.
 	Furthermore, there are at most $n!\leq n^n$ labelings of a~given unlabeled graph. Thus, the number of unlabeled $n$-vertex  $(cf)$-good graphs is bounded from below by
 	\begin{align*}  (1-o(1))\cdot \frac{1}{n^n}\cdot  2^{\frac{\gamma }{2}\cdot (n-1)f(n)\cdot (\log n  - \log (\gamma f(n)))} &=    2^{\frac{\gamma }{2}\cdot nf(n)\cdot (\log n  - \log (  f(n)) - O(1))}\\  
 		&\geq    2^{\frac{\gamma }{2}\cdot nf(n)\cdot (\log n  - (1-\delta) \log(n) - O(1))}\\  
&=  		2^{(\delta \gamma /2-o(1))\cdot n f(n) \log n},
 	        \end{align*}as claimed, since $\log n \leq f(n)\leq Cn^{1-\delta} $ by \Slin. 
 \end{proof}

 	\section{Tight bounds on labeling schemes for monotone classes}
 	\label{sec:monotone}
 We begin in \cref{sec:general-framework} with a lemma which is useful for bounding the speed when constructing monotone classes with no implicit representation. This is then used to prove our lower bound in \cref{subsec:lower}. Finally, in \cref{sec:monotone-upper} we give a matching upper bound on labeling schemes for monotone classes, this follows from \cite{KNR92} and included mainly for completeness.

 	\subsection{Construction of monotone classes}
 	\label{sec:general-framework}
 We begin with a lemma showing that, for a decent function $f$, we can create monotone classes from the union of many $f$-good graphs and still maintain control over the speed. The proof follows the broad idea of \cite[Claim 3.1]{HH22}.

 	\begin{lemma}\label{lem:good-tiny} 
			Let $f : \Rrr \rightarrow \Rrr$ be $(\delta,C,s)$-\dec  for some constants $\delta \in (0,1)$, $C\geq 1$, and $s\geq 2$.  
		Let $c > 0$ be a constant, and, for every $n \in \Nn$, let $M_n$ be any set of $(cf)$-good unlabeled $n$-vertex graphs satisfying 
		$|M_n| \leq  \big\lceil 2^{\sqrt{nf(n)}} \big\rceil$.
 		Then the speed of $\Xc := \mon(\cup_{n \in \Nn} M_n)$ is $2^{O(nf(n))}$.
 	\end{lemma}
 	\begin{proof}
 	Let $\Yc := \her(\cup_{n \in \Nn} M_n)$. Note that $\Xc = \mon(\Yc)$.
 	We first estimate the speed of $\Yc$.
 	For an $n$-vertex graph $G\in \Yc$, let $N$ be the smallest integer such that $G$ is an induced subgraph of a graph $H \in M_N$. We split the proof over two cases: $(i)$: $N \geq n^2$, and $(ii)$: $N < n^2$. 
 	\begin{labeling}{\textit{Case (ii)}:}
 		\item[\textit{Case (i)}:] Since $H$ is a $(cf)$-good $N$-vertex graph and $G$ is its $n$-vertex induced subgraph, where $n \leq \sqrt{N}$, it follows from  \cref{def:good} that $G$ must have at most $g(n): = cnf(n)/\log n$ many edges. The number of such graphs is at most  
 		\[
 		\binom{\binom{n}{2}}{g(n) } \leq \left( \frac{n^2 e}{ g(n)} \right)^{g(n)}= 	2^{g(n)\cdot \log \frac{n^2e}{g(n)}} = 2^{c\frac{nf(n)}{\log n}\cdot \log \frac{n^2e}{g(n)}} = 2^{O(nf(n))},
 		\] 
 		and so $\Yc$ contains $2^{O(nf(n))}$ many $n$-vertex labeled graphs 
 		each of which is an induced subgraph of a graph in $M_N$ for some $N$ with $n \leq \sqrt{N}$.
 		\item[\textit{Case (ii)}:] For this case, we simply use the fact that any
 		$H \in M_N$ has at most $N^n$ many $n$-vertex induced subgraphs.
 		Thus, the number of $n$-vertex labeled graphs in $\Yc$
 		each of which is an induced subgraph of a graph in $M_N$ for some $N$ with $N < n^2$ is bounded from above by
 		\begin{align*}
 			n! \cdot \sum_{N=n}^{n^{2}} N^n\cdot  |M_N| 
 			& \leq 
 			n! \cdot \sum_{N=n}^{n^{2}} N^n\cdot  \left\lceil 2^{\sqrt{Nf(N)}} \right\rceil  \\
 			& \leq 
 			n! \cdot n^{2} \cdot (n^{2})^n \cdot\left\lceil  2^{\sqrt{n^2f(n^{2})}}\right\rceil \\
 			& \leq 2^{O(n \log n)} \cdot\left\lceil  2^{\sqrt{C}nf(n)}\right\rceil\\
 			&=  2^{O(nf(n))},
 		\end{align*}
 		where in the last inequality we used \Smul of $f$, and in the final equality we used the fact that $f(x)\geq \log x$.
 	\end{labeling} 
 	 Thus, $|\Yc_n| = 2^{O(nf(n))}$. 
 	 Now, since every $n$-vertex labeled graph in $\Xc$ is a subgraph of an $n$-vertex labeled
 	 graph in $\Yc$, and, due to $(cf)$-goodness, every graph in $\Yc_n$
 	 has at most $2^{cnf(n)}$ $n$-vertex subgraphs, we conclude that
 	 $|\Xc_n| \leq |\Yc_n| \cdot 2^{cnf(n)} = 2^{O(nf(n))}$.
 	\end{proof}

	\subsection{Lower bound}\label{subsec:lower}

		We can now show the main result of the paper, which we recall for convenience.
	
	{\renewcommand{\thetheorem}{\ref{th:main}}
		\begin{theorem}
			\lowerbound	
	\end{theorem}}
	\addtocounter{theorem}{-1}
	\begin{proof}	
By assumption $f : \Rrr \rightarrow \Rrr$ is $(\delta,C,s)$-\dec  for some constants $\delta \in (0,1)$, $C\geq 1$, and $s\geq 2$. 	We will construct a monotone class (via the probabilistic method) with the speed $2^{O(n f(n))}$ that does not admit a universal graph of size $u_n := 2^{ f(n)\log n}$.  Fix $\gamma := 4/\delta >1$ and let $c:=c(\gamma, \delta,C,s)>0$ be the satisfying constant from \cref{th:Gnp-sparsness} corresponding to this choice of $\gamma$. Let $k_n := \left\lceil 2^{\sqrt{nf(n)}} \right\rceil$.
		
		The number of distinct $u_n$-vertex graphs is at most $2^{u_n^2}$ and 
		the number of $n$-vertex induced subgraphs of a~fixed $u_n$-vertex graph 
		is at most $\binom{u_n}{n}$.
		Hence the number of collections of $k_n$~graphs on $n$ vertices that are
		induced subgraphs of a~$u_n$-vertex (universal) graph is at most 
		\begin{equation}\label{eq:cols} 
			2^{u_n^2}\cdot \binom{\binom{u_n}{n}}{k_n}\leq  2^{u_n^2}\cdot u_n^{k_n \cdot n}.  
		\end{equation}

		On the other hand, from \cref{lem:superfactorial}, the number of different collections of $n$-vertex $(cf)$-good graphs of cardinality $k_n$ is at least 
		\begin{equation}\label{eq:unlabbelled}
			\binom{2^{(\gamma \delta /2-o(1))\cdot n f(n) \log n}}{k_n} 	\geq \left(\frac{2^{(\gamma \delta /2-o(1))\cdot n f(n) \log n}}{k_n} \right)^{k_n} = 2^{k_n \cdot (\gamma \delta/2-o(1))\cdot n f(n) \log n},
		\end{equation} 
		as $\log k_n = O(\sqrt{nf(n)})  = o(n f(n)  \log n)$. 
		By taking logarithms, we can see
		that for sufficiently large~$n$ the upper bound \eqref{eq:cols} is smaller than the lower bound \eqref{eq:unlabbelled}. In particular, taking the logarithm of \eqref{eq:cols} gives 
		
		\begin{align*}
			\log\left(2^{u_n^2}\cdot u_n^{k_n \cdot n} \right) 
			& = u_n^2 + k_n\cdot n  \log u_n \notag \\
			&= 2^{2 f(n) \log n} + k_n \cdot n f(n)\log n\\
			&= (1 +o(1))\cdot k_n \cdot n f(n) \log n, 
			\intertext{as $k_n:=\left\lceil 2^{\sqrt{nf(n)}} \right\rceil = \omega(2^{2 f(n)\log n})$.  
			However, since $\gamma =4/\delta$, the logarithm of \eqref{eq:unlabbelled} is}
			\log\left(2^{k_n \cdot (\gamma \delta/2-o(1))\cdot n f(n) \log n}  \right) &= k_n \cdot (\gamma \delta/2-o(1))\cdot n f(n) \log n\\
			&=  (2-o(1))\cdot k_n \cdot n f(n) \log n.
		\end{align*}  
		
		Thus, for any sufficiently large $n$, there exists a~collection $M_n$ of $k_n$ $(cf)$-good $n$-vertex graphs that are not representable by any universal graph
		of size at most $u_n=2^{f(n) \log n}$.	 
		Consequently, by \cref{lem:good-tiny}, the speed of $\Xc := \mon(\cup_n M_n)$ is $|\Xc_n| = 2^{O(nf(n))}$ and
		$\Xc$ does not admit a universal graph of size at most $2^{ f(n)\log n}$.
	\end{proof}

	\subsection{Upper bound}
	\label{sec:monotone-upper}

	In this section we prove the following upper bound on labeling schemes for monotone classes.
		
	{\renewcommand{\theproposition}{\ref{lem:monotone-factorial}}
		\begin{proposition}
			\monofactorial
	\end{proposition}}
	\addtocounter{proposition}{-1}
	
	\begin{proof}
			Let $\Xc$ be a~monotone class with at most $2^{Cnf(n)}$ labeled $n$-vertex graphs for every $n$.
			If an $n$-vertex graph $G \in \Xc$ has $m$ edges, then $\Xc$ contains at least $2^m$ labeled $n$-vertex graphs, as every subgraph of $G$ also belongs to $\Xc$ due to monotonicity.
			
			This implies that every $n$-vertex graph $G$ in $\Xc$ contains at most $C n f(n)$ edges, and hence, has a vertex of degree at most $2C f(n)$. Due to monotonicity of $f$, the same is true for every subgraph of $G$. Indeed, if $H$ is a $k$-vertex subgraph
			of $G$, then, since $H$ belongs to $\Xc$, the number of edges in $H$ is at most $C k f(k) \leq C k f(n)$, and therefore
			$H$ has a vertex of degree at most $2C f(n)$.	
			Thus, every $n$-vertex graph in $\Xc$ is $2C f(n)$-degenerate, and \cref{lem:degenlabel} implies that $\Xc$ admits a~$(2C f(n)+1)\cdot\lceil \log n\rceil$-bit labeling scheme.\end{proof}

	\subsection{Complexity of monotone classes}\label{sec:complexity}
	The following result shows that monotone classes are complex in the sense that they cannot be ``described'' by even a countable number of classes of a slightly larger speed. The proof of this theorem follows the exact same idea as \cite[Lemma 2.4]{Chandoo23}, also see \cite[Theorem 1.2]{BDSZZ23} for the proof of a similar theorem in the context of small classes. 
	
	{\renewcommand{\thetheorem}{\ref{Thm:complex}}
		\begin{theorem}
			\complex
	\end{theorem}}
	\addtocounter{theorem}{-1}

	\begin{proof}

		By assumption $f : \Rrr \rightarrow \Rrr$ is $(\delta,C,s)$-\dec  for some constants $\delta \in (0,1)$, $C\geq 1$, and $s\geq 2$. By \cref{lem:superfactorial} there exists some $c:=c(\delta,C,s)>0$ such that if $\mathcal{G}_n$ is the number of unlabeled $(cf)$-good $n$-vertex graphs, then $|\mathcal{G}_n|\geq 2^{(2-o(1))\cdot nf(n)\log n}$ for every $n \in \Nn$. Let $k_n:=\big\lceil 2^{\sqrt{nf(n)}} \big\rceil$. Then, by \cref{lem:good-tiny}, for every $n \in \Nn$ and $M_n \subseteq \mathcal{G}_n$  satisfying $|M_n| \leq k_n$,  the speed of $ \Xc := \mon(\cup_{n \in \Nn} M_n)$ is $2^{O(nf(n))}$. Thus if we wish to build such a class $\Xc$ there are \begin{equation}\label{eq:numsets}\binom{|\mathcal{G}_n|}{k_n} \geq \left(\frac{|\mathcal{G}_n|}{k_n}\right)^{k_n}  \geq 2^{(2-o(1))\cdot k_n\cdot  nf(n)\log n} ,\end{equation} ways of selecting the set $ M_n\subseteq \mathcal{G}_n$.    
		
		Let $\mathbb{X}=(\mathcal{D}^i)_{i\in \mathbb{N}}$ be any countable collection of classes, satisfying $|\mathcal{D}^i_n|\leq 2^{nf(n)\log n} $ for each $n\in \Nn$ and $i\in \Nn$. Any class $\mathcal{D}^i \in \mathbb{X}$ contains at most $2^{k_n\cdot nf(n)\log n}$ different sets of $n$-vertex graphs with size $k_n$. By \eqref{eq:numsets}, there is some constant $N_0$ such that for all $n\geq N_0$  this is less than the number of choices of sets $M_n \subseteq \mathcal{G}_n$ with size $k_n$. Thus, given any such set $\mathbb{X}$, for each $n\in \Nn$ we can take some $M_n\subseteq \mathcal{G}_{n+N_0}$ such that $M_n\not\subseteq \mathcal{D}^n$ and thus $\Xc \notin \mathbb{X}$.\end{proof}

\iftoggle{anonymous}{
}{%
	\bigskip
	\textbf{Acknowledgments.}
	We are grateful to Nathan Harms for valuable feedback on the early version of this paper. 
	This work has been supported by Research England funding to enhance research culture, by the Royal Society (IES\textbackslash R1\textbackslash 231083), by the ANR projects TWIN-WIDTH (ANR-21-CE48-0014) and Digraphs (ANR-19-CE48-0013), and also the EPSRC project EP/T004878/1: Multilayer Algorithmics to Leverage Graph Structure.
}

\bibliographystyle{alpha}
\bibliography{biblio}
	
\end{document}